\documentclass{article}

\usepackage[utf8]{inputenc}
\usepackage{amsmath, amssymb, amsthm, bm}
\usepackage{graphicx, booktabs, subfigure}
\usepackage{hyperref}

\newtheorem{lemma}{Lemma}
\newtheorem{theorem}{Theorem}
\newtheorem{remark}{Remark}
\newtheorem{example}{Example}

\graphicspath{{Figure/}}

\title{A Fourier-Bessel method with a regularization strategy for the boundary value problems of the Helmholtz equation}

\author{
Deyue Zhang\footnote{School of Mathematics, Jilin University, Changchun,
P. R. China. {\it dyzhang@jlu.edu.cn}}, \ 
Fenglin Sun\footnote{School of Mathematics, Jilin University, Changchun,
P. R. China.},\ 
Yan Ma\footnote{School of Mathematics, Jilin University, Changchun,
P. R. China.} \ and
Yukun Guo\footnote{Department of Mathematics, Harbin Institute of Technology, Harbin, P. R. China. {\it ykguo@hit.edu.cn} (Corresponding author)}
}

\date{}

\begin{document}

\maketitle

\begin{abstract}
This paper is concerned with the Fourier-Bessel method for the boundary value problems of the Helmholtz equation in a smooth simply connected domain. Based on the denseness of Fourier-Bessel functions, the problem can be approximated by determining the unknown coefficients in the linear combination. By the boundary conditions, an operator equation can be obtained. We derive a lower bound for the smallest singular value of the operator, and obtain a stability and convergence result for the regularized solution with a suitable choice of the regularization parameter. Numerical experiments are also presented to show the effectiveness of the proposed method.
\end{abstract}

\section{Introduction}

The boundary value problems (BVP) of the Helmholtz equation appear in many scientific fields and engineering applications, such as wave propagation, vibration, electromagnetic scattering and so on. The properties of the solution to the BVP have been studied widely, and many numerical methods have been proposed to solve the BVP, such as the finite element method \cite{AB84, Ciarlet78, Rukavishnikov13}, the finite difference method \cite{MG80, Thomas95}, the boundary integral equation method \cite{Chen,Colton13, Kress}, and etc.

This paper concerns with the BVP of the Helmholtz equation in a smooth simply connected domain, and a Fourier-Bessel method (FBM) is considered to solve the problems. The FBM was the first to be presented in
\cite{Zhang, Liu} to solve the Cauchy problems for the Helmholtz equation. The main idea is to approximate the exact solution by a linear combination of the Fourier-Bessel functions. Here, we make use of the idea to
solve the BVP of the Helmholtz equation. And the problems are approximated by determining the unknown
coefficients in the linear combination. By using the boundary conditions, an operator equation is easily obtained, which can be solved by a regularization method since the operator is compact and injective.

The main purpose of this paper is to provide a stability analysis of the FBM for solving the BVP. In this paper, we propose an approach to derive a lower bound for the smallest singular value of the operator,
and obtain a stability and convergence result for the regularized solution with a suitable choice of the regularization parameter. We emphasize that our idea does work for arbitrary domains with smooth boundaries, and the idea can be generalized to other equations.

This paper is organized as follows. In Section 2, we present the Fourier-Bessel approximation of the solution to the BVP, and an operator equation for the coefficients. In Section 3, we derive a lower bound for the smallest singular value of the operator, and solve the equation by the Tikhonov regularization method. A convergence and stability result is obtained with a suitable choice of the regularization parameter. Finally, several numerical examples are included to show the effectiveness of our method.


\section{The harmonic polynomial method}

Let $D\subset\mathbb{R}^2$ be an open, bounded and simply connected domain with a $C^{\infty}$ boundary $\Gamma$ (see \cite{Evans}) .

Consider the following BVP: Given $f\in L^2(\Gamma)$, find $u$ such that $u$ satisfies
\begin{align}
&\Delta u+k^2u=0, \quad\text{in}\  D \label{eqnlaplace}\\
&\frac{\partial u}{\partial \nu}+\mathrm{i} k u=f, \quad \text{on}\ \Gamma \label{BoundaryR}
\end{align}
where $k>0$ is the wavenumber and $\nu$ is the unit normal to the boundary $\Gamma$ directed into the exterior of $D$.

Now, we review the FBM. Recall that for $n\in \mathbb{Z}$ the Fourier-Bessel functions are
\begin{equation}\label{FB}
\varphi_n(x):= \frac{2^{|n|}|n|!}{k^{|n|}M^{|n|}}J_n(k r) \mathrm{e}^{\mathrm{i} n\theta},
\end{equation}
where $J_n(t)$ is the Bessel function of the first kind of order $n$, under the polar coordinates $(r,\theta): x=(r\cos \theta, r\sin\theta)$, and the constant $M>r_D=\max\limits_{x\in \overline{D}}|x|$. Then, based on the idea of the FBP, an approximate solution $u_N$ for the BVP \eqref{eqnlaplace}-\eqref{BoundaryR} can be expressed by the following linear combination
\begin{equation}\label{FBSolution}
u_N(x)=\sum_{n=-N}^N c_n\varphi_n(x),
\end{equation}
where $c_n$ ($n\in\mathbb{Z}$) are constants.

To determine the parameters $c_n$, by using the boundary conditions, we derive and solve the following equations
\begin{equation}\label{Eqn}
A_N \mathbf{c}_N=f, \quad \text{on}\ \Gamma,    
\end{equation}
where $\mathbf{c}_N \in \mathbb{C}^{2N+1}$, and the trace operator $A_N:\mathbb{C}^{2N+1}\to L^2(\Gamma)$ is defined by
\begin{equation}\label{Operator}
(A_N \mathbf{c}_N)(x):=\mathrm{i} k \sum\limits_{n=-N}^N c_n\varphi_n(x)+\dfrac{\partial}{\partial
\nu}\sum\limits_{n=-N}^N c_n\varphi_n(x), \quad x \in \Gamma.
\end{equation}

In the following, we present the approximation result of Fourier-Bessel functions and the property of the operator $A_N$.

In the paper \cite{Zhang}, we have proved the following two lemmas.

\begin{lemma}\label{lemma 2.1}
Let $u\in H^{3/2}(D)$ satisfy the Helmholtz equation. Let $B$ be a bounded and simply connected domain with  $\partial B \in C^2$ such that $\overline{D}\subset\subset B$. Then for every $\varepsilon>0$, there exists a single-layer potential $v_{\psi}$ of the form
\begin{equation}\label{potential}
v_{\psi}(x):=\int_{\partial B}\Phi(x,y)\psi(y)\mathrm{d} s(y), \quad x \in B
\end{equation}
for some $\psi\in L^2(\partial B)$, where $\Phi(x,y)$ is the fundamental solution to the Helmholtz equation, such that
\begin{equation*}
  \|v_{\psi}-u\|_{H^{3/2}(D)}\leq \varepsilon,
\end{equation*}
and especially
\begin{equation*}
  \|v_{\psi}-u\|_{L^2(\Gamma)}+
  \left\|\frac{\partial v_{\psi}}{\partial \nu}-\frac{\partial u}{\partial \nu}\right\|_{L^2(\Gamma)}
  \leq \varepsilon.
\end{equation*}
\end{lemma}

\begin{lemma}\label{lemma2.2}
Assume that $v$ is the solution of the Helmholtz equation in $B_\rho:=\{\bm{x}\in \mathbb{R}^2: |\bm{x}|<\rho\}$, then there exists a sequence $\{u_N^*\}_{N=0}^{\infty}$ of Fourier-Bessel functions of the form \eqref{FBSolution}, such that
\begin{align*}
    & \|v-u_N^*\|_{L^{\infty}(B_r)} \leq C_1 k \tau^{-N}, \\
    & \|\nabla(v-u_N^*)\|_{L^{\infty}(B_r)} \leq C_2 k\tau^{-N},
\end{align*}
where $r<R<\rho, \tau=R/r>1$, and the constants $C_1$ and $C_2$ are independent of $N$ and $k$.
\end{lemma}

From Lemma \ref{lemma 2.1} and Lemma \ref{lemma2.2}, we obtain the following approximation result.
\begin{theorem}\label{theorem 2.1}
Let $u\in H^{3/2}(D)$ satisfy the Helmholtz equation. Then for every $\varepsilon>0$, there exists a sequence
$\{u_N^*\}_{N=0}^{\infty}$ of Fourier-Bessel functions of the form \eqref{FBSolution}, and a constant $C_0$ independent of $N$ and $k$, such that
\begin{equation}\label{Thm2.1Err1}
    \|u_N^*-u\|_{H^{3/2}(D)}\leq C_0 k^2\tau^{-N}+\varepsilon,
\end{equation}
and
\begin{equation}\label{Thm2.1Err2}
    \left\|\frac{\partial (u_N^*-u)}{\partial \nu} +\mathrm{i} k(u_N^*-u)\right\|_{L^2(\Gamma)}
    \leq C_0 k^2\tau^{-N}+\varepsilon.
\end{equation}
\end{theorem}

\begin{proof}
Choose $B_r$, $B_R$, $B_\rho$($r<R<\rho$) and $B$ such that $\overline{D}\subset \subset B_r$ and $\overline{B_\rho}\subset \subset B$, where $B$ is a bounded and simply connected domain with $\partial B \in C^2$.

From Lemma \ref{lemma 2.1}, it can be seen that for every $\varepsilon>0$, there exists a single-layer potential $v_{\psi}$ of the form \eqref{potential} for some $\psi \in L^2(\partial B)$, such that
\begin{equation}\label{Thm2.1Err3}
  \|v_{\psi}-u\|_{H^{3/2}(D)}\leq \varepsilon.
\end{equation}

By using Lemma \ref{lemma2.2} for $v_{\psi}$ in $B_\rho$, we know that there exists a sequence $\{u_N^*\}_{N=0}^{\infty}$ of Fourier-Bessel functions, such that
\begin{align*}
 & \|u_N^*-v_{\psi}\|_{L^{\infty}(B_r)}\leq C_1 k\tau^{-N}, \\
 & \|\nabla(u_N^*-v_{\psi})\|_{L^{\infty}(B_r)}\leq C_2 k \tau^{-N},
\end{align*}
which yield
\begin{align*}
 & \|u_N^*-v_{\psi}\|_{L^{2}(\Gamma)}\leq C_3 k\tau^{-N},\\
 & \left\|\frac{\partial(u_N^*-v_{\psi})}{\partial\nu}\right\|_{L^2(\Gamma)}\leq C_4 k\tau^{-N}.
\end{align*}
By using the triangle inequality, we obtain the estimate \eqref{Thm2.1Err2}.

Let $w=u_N^*-v_{\psi}$. Then, we have $\Delta w+k^2w=0$ in $D$, and thus,
\begin{align*}
 k^2\|w\|_{L^2(D)}^2= & k^2\int_D|w|^2\mathrm{d} x \\
 = & \int_D|\nabla w|^2\mathrm{d} x - \int_\Gamma \frac{\partial w}{\partial \nu} \bar{w}\mathrm{d} s \\
 \leq & \|\nabla w\|_{L^2(D)}^2 + \|w\|_{L^2(\Gamma)}
 \left\|\frac{\partial w}{\partial \nu}\right\|_{L^2(\Gamma)}.
\end{align*}
Further, from $w \in H^{3/2}(D)$ (see \cite[Theorem 7.8.3]{Chen}), and the interior regularity results of the Laplace
equation (see \cite[Theorem 6.7.3]{Chen}), it follows that
\begin{align*}
 \|w\|_{H^{3/2}(D)}\leq & C_5 k^2\|w\|_{L^2(D)}+C_6\left\|\frac{\partial w}{\partial
 \nu}+w\right\|_{L^2(\partial D)} \\
 \leq & C_5 k\left(\|\nabla w\|_{L^2(D)}+\|w\|_{L^2(\Gamma)}^{1/2}
 \left\|\frac{\partial w}{\partial \nu}\right\|_{L^2(\Gamma)}^{1/2}\right)+C_7 k\tau^{-N} \\
 \leq & C_0 k^2\tau^{-N},
\end{align*}
which, together with the triangle inequality, leads to \eqref{Thm2.1Err1}.
\end{proof}

%
%
%

For convenience, we denote
\begin{equation}\label{HPSolution*}
u_N^*(x)=\sum\limits_{n=-N}^{N} c_n^*\varphi_n(x),
\end{equation}
and $\mathbf{c}_N^*=(c_0^*, c_1^*, c_{-1}^*, \cdots, c_N^*, c_{-N}^*)^\top \in \mathbb{C}^{2N+1}$. And from Theorem \ref{theorem 2.1}, we have
\begin{equation}\label{EqnC*}
A_N \mathbf{c}_N^*=f_N^{\varepsilon}, \quad \text{on}\ \Gamma,
\end{equation}
where $\|f_N^{\varepsilon}-f\|_{L^2(\Gamma)} \leq C_0 k^2\tau^{-N}+\varepsilon$.

\begin{lemma}\label{lemma2.4}
The operator $A_N:\mathbb{C}^{2N+1}\to L^2(\Gamma)$ defined by \eqref{Operator} is compact and injective.
\end{lemma}

\begin{proof}
It is sufficient to prove that $A_N:\mathbb{C}^{2N+1}\to  L^2(\Gamma)$ is compact.
From Theorem 2.19 in \cite{Kress} we know that the identity operator $I:\mathbb{C}^{2N+1}\to \mathbb{C}^{2N+1}$ is compact since $\mathbb{C}^{2N+1}$ has finite dimension. Let $\{\mathbf{c}^{l}\}_{l=1}^{\infty}$ be a bounded sequence
in $\mathbb{C}^{2N+1}$. By using Theorem 2.12 in \cite{Kress} for the operator $I$, we have that the sequence $\{\mathbf{c}^{l}\}_{l=1}^{\infty}$ contains a convergent subsequence $\{\mathbf{c}^{l_j}\}_{j=1}^{\infty}$, i.e.,
there exists a $\tilde{\mathbf{c}}\in \mathbb{C}^{2N+1}$ such that $\lim\limits_{j\to \infty}|\mathbf{c}^{l_j}-\tilde{\mathbf{c}}|=0$. Further, by using
\begin{align*}
J_n(z)=& \frac{2\left(\frac{z}{2}\right)^n}{\sqrt{\pi}\,\Gamma(n+\frac{1}{2})}\int_0^1(1-t^2)^{n-\frac{1}{2}}\cos(z t)\,\mathrm{d} t \\
=& \frac{z^n}{2^n n!}\frac{(2n)!!}{(2n-1)!!}\frac{2}{\sqrt{\pi}}\int_0^1(1-t^2)^{n-\frac{1}{2}}\cos(z t)\mathrm{d} t
\end{align*}
and the formula $J'_n(t)=n J_n(t)/t-J_{n+1}(t)$ ($n\geq0$), we have for $x\in \overline{D}$ and $n\geq0$,
\begin{equation*}
  \left|J_n(k r)\right|\leq \frac{k^n r^n}{2^n n!}\frac{(2n)!!}{(2n-1)!!}\frac{2}{\sqrt{\pi}}\int_0^1(1-t^2)^{-\frac{1}{2}}\,\mathrm{d} t
  = K_1 n\frac{k^n r^n}{2^n n!},
\end{equation*}
and
\begin{equation*}
 \left|\nabla J_n(k r)\right|\leq K_2 (1+k^2r^2)n^2 \frac{k^n r^{n-1}}{2^n n!},
\end{equation*}
where the constants $K_1$ and $K_2$ are independent of $N$ and $k$.
Therefore, we have
\begin{align*}
 \left\|\varphi_n\right\|_{L^2(\Gamma)}^2+\left\|\frac{\partial\varphi_n}{\partial\nu}\right\|_{L^2(\Gamma)}^2
 \leq K_3(1+k^4) n^4 q^{2n},
\end{align*}
where $q=r_D/M<1$,  $n\geq0$, and the constant $K_3$ is independent of $N$ and $k$. Now, with the help of the Cauchy inequality and $J_{-n}(t)=(-1)^n J_n(t)$,  we obtain that
\begin{align*}
  \left\|A_N\mathbf{c}^{l_j}-A_N\tilde{\mathbf{c}}\right\|_{L^2(\Gamma)}^2\leq & \left|\mathbf{c}^{l_j}-\tilde{\mathbf{c}}\right|^2(1+k^2)
  \sum\limits_{n=-N}^{N}\left(\left\|\varphi_{n}\right\|_{L^2(\Gamma)}^2
 +\left\|\frac{\partial \varphi_n}{\partial\nu}\right\|_{L^2(\Gamma)}^2\right) \\
 \leq &\left|\mathbf{c}^{l_j}-\tilde{\mathbf{c}}\right|^2 K(1+k^2)(1+k^4)
   \sum\limits_{n=0}^{\infty}n^4q^{2n},
\end{align*}
where the constant $K$ is independent of $N$ and $k$, and thus
\begin{equation*}
\lim\limits_{j\to\infty}\left\|A_N\mathbf{c}^{l_j}-A_N\tilde{\mathbf{c}}\right\|_{L^2(\Gamma)}=0
\end{equation*}
which means that the sequence $\left\{A_N\mathbf{c}^{l}\right\}_{l=1}^{\infty}$ contains a convergent subsequence $\left\{A_N\mathbf{c}^{l_j}\right\}_{j=1}^{\infty}$. Therefore, the operator $A_N$ is compact by Theorem 2.12 in \cite{Kress}.

Next, let $A_N \mathbf{c}_N=0$. This means that there exists a function $g_N=\sum\limits_{n=-N}^{N} c_n\varphi_n(x)$ satisfying the Helmholtz equation in $D$ and the boundary condition
$\frac{\partial g_N}{\partial\nu} +\mathrm{i} k g_N=0$ on $\Gamma$.  From the uniqueness of solution to the BVP and the analyticity of the Fourier-Bessel functions, it can be seen that $g_N=0$ in $\mathbb{R}^2$ and then $\mathbf{c}_N=0$. Therefore the operator $A_N$ is injective.
\end{proof}

\begin{remark}
In general, equation \eqref{Eqn} is an operator equation of the first kind which cannot be solved directly, since from Lemma \ref{lemma2.4}, the trace operator $A_N$ is compact, and we don't know if the function $f$ is in the range $A_N(\mathbb{C}^{2N+1})$ of $A_N$. Therefore, we will solve the ill-posed operator equation \eqref{Eqn} by a regularization method in the next section, and then give an error estimate.
\end{remark}


\section{A regularization method for solving the equations}

Due to the ill-posedness, we consider the perturbed equations
\begin{equation}\label{Eqnperturbed}
A_N \mathbf{c}_N^{\delta}=f^{\delta},
\end{equation}
where $f^{\delta}\in L^{2}(\Gamma)$ are measured noisy data satisfying $ \|f-f^{\delta}\|_{L^2(\Gamma)}\leq\delta\|f\|_{L^2(\Gamma)}$ with $0<\delta<1$.

A regularized solution to \eqref{Eqnperturbed} is a linear combination of harmonic polynomials
\begin{equation}\label{RHPSolution}
  u_N^{\alpha,\delta}(x):= \sum\limits_{n=-N}^{N} c_n^{\alpha,\delta}\varphi_n(x),
\end{equation}
where the coefficients $c_n^{\alpha,\delta}$ are determined by solving the following equation:
\begin{equation}\label{REqnperturbed}
\alpha \mathbf{c}_N^{\alpha,\delta}+A_N^{\ast} A_N\mathbf{c}_N^{\alpha,\delta}=A_N^{\ast} f^{\delta}.
\end{equation}

Before considering the error estimate, we try to find a lower bound for the smallest singular value of the operator $A_N$.

Let the singular system of $A_N$ be $(\mu_j, \mathbf{d}_j,\phi_j), j=1,\cdots,2N+1$, and $\mu_{\min}=\min\{\mu_1,\cdots,\mu_{2N+1}\}$. Let the origin $O$ be located inside $D$, $r_{\rm in}^{\max}=\max\limits_{B_t\subseteq D}t$ and $r_{\rm ex}^{\min}=\min\limits_{D\subseteq B_t}t$.
Take $r_{\rm in}=\min\{r_{\rm in}^{\max}, k^{-1}\}$ and $r_{\rm ex}>r_{\rm ex}^{\min}$ such that
\begin{equation}\label{Condition}
\overline{B}_{r_{\rm in}}\subsetneqq D, \quad \overline{D}\subsetneqq B_{r_{\rm ex}}.
\end{equation}
Then, the following estimate holds.
\begin{lemma}\label{lemma3.1}
For $r_{\rm in}=\min\{r_{\rm in}^{\max}, k^{-1}\}$, we have
\begin{equation}\label{JnEstimate}
  |J_n(k r_{\rm in})|\geq \frac{3}{4} \frac{k^n r_{\rm in}^n}{2^n n!}, \quad n=0,1,\cdots.
\end{equation}
\end{lemma}

\begin{proof} From $r_{\rm in}=\min\{r_{\rm in}^{\max}, k^{-1}\}$ and the definition of the Bessel functions of the first kind, we have $0<k r_{\rm in}/2\leq 1/2$, and
\begin{align*}
J_n(k r_{\rm in})= & \sum_{p=0}^{\infty}\frac{(-1)^p}{p!(n+p)!}\left(\frac{k r_{\rm in}}{2}\right)^{n+2p} \\
= & \frac{k^n r_{\rm in}^n}{2^n n!}\left(1-\frac{\left(\frac{k r_{\rm in}}{2}\right)^{2}}{n+1}
+\sum_{p=2}^{\infty}(-1)^p\frac{\left(\frac{k r_{\rm in}}{2}\right)^{2p}}{p!(n+1)\cdots(n+p)}\right) \\
\geq & \frac{k^n r_{\rm in}^n}{2^n n!}\left(1-\frac{\left(\frac{k r_{\rm in}}{2}\right)^{2}}{n+1}\right) \\
\geq & \frac{k^n r_{\rm in}^n}{2^n n!} \frac{3}{4}, \qquad n=0,1,\cdots,
\end{align*}
which completes the proof.
\end{proof}

Take $M=r_{\rm ex}$. Then, from Lemma \ref{lemma3.1} we derive the following result.
\begin{theorem}\label{theorem3.1}
There exists a positive constant $c$ independent of $N$ and $k$, such that
\begin{equation}\label{Eigenvalue}
\mu_{\min} \geq\frac{ c \min\{1,k\}}{1+\sqrt{k}} \left(\frac{r_{\rm in}}{r_{\rm  ex}}\right)^N.
\end{equation}
\end{theorem}
\begin{proof}
It is clear that $\forall \mathbf{c}_N \in \mathbb{C}^{2N+1}$,
$$
\Delta u_N+k^2u_N=0,\quad \|A_N \mathbf{c}_N\|_{L^2(\Gamma)} =
\left\|\frac{\partial u_N}{\partial \nu}+\mathrm{i} k u_N\right\|_{L^2(\Gamma)}.
$$
From the trace theorem and the interior regularity results of the Helmholtz equations (\cite[Theorem 1.8]{Baskin}), it follows that
\begin{equation}\label{regularity}
\|A_N \mathbf{c}_N\|_{L^2(\Gamma)}\geq C_1^\prime\min\{1,k\} \| u_N\|_{H^1(D)}\geq  C_2^\prime \min\{1,k\} \|u_N\|_{L^2(\partial B_{r_{\rm in}})},
\end{equation}
where the constants $C_1^\prime$ and $C_2^\prime$ are independent of $N$ and $k$.

In the following, we compute $\|u_N\|_{L^2(\partial B_{r_{\rm in}})}$. It can be readily seen that
\begin{equation*}
u_N(x)=\sum_{n=-N}^{N}c_n\frac{2^{|n|}|n|!}{k^{|n|}r_{\rm ex}^{|n|}}J_n(k r_{\rm in}) \mathrm{e}^{\mathrm{i} n\theta}, \quad \text{for}\ x\in \partial B_{r_{\rm in}}.
\end{equation*}
And thus
\begin{equation*}\label{Thm3.1UN}
\|u_N\|^2_{L^2(\partial B_{r_{\rm in}})}=2\pi r_{\rm in}
 \sum_{n=-N}^N \left(\frac{2^{|n|}|n|!}{k^{|n|}r_{\rm ex}^{|n|}}\right)^2J_n^2(k r_{\rm in})c_n^2.
\end{equation*}
By using Lemma \ref{lemma3.1} and $J_{-n}(t)=(-1)^n J_n(t)$, we deduce that
\begin{equation}\label{Thm3.1UN1}
\|u_N\|_{L^2(\partial B_{r_{\rm in}})}\geq
  \frac{3}{4}\sqrt{2\pi r_{\rm in}}\left(\frac{r_{\rm in}}{r_{\rm ex}}\right)^N|\mathbf{c}_N|.
\end{equation}
From $r_{\rm in}=\min\{r_{\rm in}^{\rm max}, k^{-1}\}$, we see that $r_{\rm in}\geq \frac{r_{\rm in}^{\rm max}}{k r_{\rm in}^{\rm max}+1}$,
which, together with \eqref{regularity} and \eqref{Thm3.1UN1}, leads to the estimate \eqref{Eigenvalue}.
\end{proof}


Now, we can achieve an error estimate presented in the following theorem.
\begin{theorem}\label{theorem3.2}
	There exists a positive constant $C$ independent of $N$ and $k$, such that
	\begin{align}
	\left\|u_N^{\alpha,\delta}-u\right\|_{H^1(D)} \leq &  C \bigg(\left(\frac{1+k^3}{\min\{1,k\}}\frac{\delta}{\sqrt{\alpha}}+\frac{1+k^4}{\min\{1,k^3\}}\sqrt{\alpha} \tau_0^{2N}\right)
	\|f\|_{L^{2}(\Gamma)} \nonumber \\
    &+\frac{1+k^{7/2}}{\min\{1,k^2\}}\tau_0^{N}(C_0k^2\tau^{-N}+\varepsilon)\bigg), \label{Thm3.2Err1}
	\end{align}
	where $\tau_0=r_{\rm ex}/r_{\rm in}$. Furthermore, let $\eta>1$, $\tau=\tau_0^{2}$, and $ \tau_{\rm min}=r_{\rm ex}^{\rm min}/r_{\rm in}^{\rm max}$. Then
	
	(a) For $0< k \leq 1$, take  $N=\eta\ln|\ln\delta|$ and choose the regularization parameter $\alpha=k^2\delta\tau_0^{-2N}$, then the following result holds
	\begin{equation}\label{Thm3.2Err2}
	\left\|u_N^{\alpha,\delta}-u\right\|_{H^1(D)}\leq C\left( k^{-2} \delta^{1/2}|\ln\delta|^{\lambda} \|f\|_{L^{2}(\Gamma)}
	+ C_0|\ln\delta|^{-\lambda} + k^{-2} |\ln\delta|^{\lambda}\varepsilon\right),
	\end{equation}
	where $\lambda=\eta\ln\tau_0$.
	
	(b) For $ k >1$, take $N=\frac{11\ln k}{2\ln \tau_{\rm min}}+\eta\ln|\ln\delta|$ and choose the regularization parameter $\alpha=k^{-1}\delta \tau_0^{-2N}$, then we have
	\begin{equation}\label{Thm3.2Err21}
	\left\|u_N^{\alpha,\delta}-u\right\|_{H^1(D)}\leq  C\left(k^{\sigma}\delta^{\frac{1}{2}}  |\ln\delta|^{\lambda}
	\|f\|_{L^{2}(\Gamma)} + C_0 |\ln\delta|^{-\lambda} + k^{\sigma}|\ln\delta|^{\lambda}\varepsilon\right),
	\end{equation}
	where 
	$$
	\sigma=\frac{7}{2}+\frac{11\ln \tau_0}{2\ln\tau_{\rm min}}.
	$$
\end{theorem}

\begin{proof}
	Similar to \eqref{regularity} and the proof of Lemma \ref{lemma2.4}, we have
	\begin{align}
	\| u_N^{\alpha,\delta}-u_N^*\|_{H^1(D)}^2 \leq & \frac{C^\prime}{\min\{1,k^2\}}\left\| \frac{\partial (u_N^{\alpha,\delta}-u_N^*)}{\partial \nu}+\mathrm{i} k (u_N^{\alpha,\delta}-u_N^*)
	\right\|_{L^{2}(\Gamma)}^2\nonumber \\ 
	\leq &\left|\mathbf{c}_N^{\alpha,\delta}-\mathbf{c}_N^*\right|^2\frac{C^\prime(1+k^2)}{\min\{1,k^2\}}
	\sum\limits_{n=-N}^{N}\left(\left\|\varphi_{n}\right\|_{L^2(\Gamma)}^2
	+\left\|\frac{\partial \varphi_{n}}{\partial \nu}\right\|_{L^2(\Gamma)}^2\right) \nonumber \\
	\leq &\frac{C^{\prime\prime}(1+k^2)(1+k^4)}{\min\{1,k^2\}}\left|\mathbf{c}_N^{\alpha,\delta}-\mathbf{c}_N^*\right|^2, \label{Thm3.2Err3}
	\end{align}
	where constants $C^{\prime}, C^{\prime\prime}$ are independent of $N$ and $k$. This means that we need to estimate $|\mathbf{c}_N^{\alpha,\delta}-\mathbf{c}_N^*|$.
	
	Based on the singular value decomposition of $A_N$, the solution $\mathbf{c}_N^*$ to \eqref{EqnC*} can be written as
	$$
		\mathbf{c}_N^*=\sum\limits_{j=1}^{2N+1}\frac{1}{\mu_j}(f_N^\varepsilon, \phi_j)\mathbf{d}_j,
	$$
	and the solution to  \eqref{REqnperturbed} is
	$$
		\mathbf{c}_N^{\alpha,\delta}=\sum\limits_{j=1}^{2N+1}\frac{\mu_j}{\alpha+\mu_j^2}(f^\delta,	\phi_j)\mathbf{d}_j,
	$$
	where $(\cdot, \cdot )$ stands for the inner product on $L^2(\Gamma)$. Then, we have
	\begin{align*}
		\mathbf{c}_N^{\alpha,\delta}-\mathbf{c}_N^*= & \underbrace{\sum\limits_{j=1}^{2N+1}\frac{\mu_j^2}{\alpha+\mu_j^2}\frac{(f^\delta-f, \phi_j)}{\mu_j}\mathbf{d}_j}_{=:I_1}
		+\underbrace{\sum\limits_{j=1}^{2N+1}\left(\frac{\mu_j^2}{\alpha+\mu_j^2}-1\right)\frac{(f, \phi_j)}{\mu_j}\mathbf{d}_j}_{=:I_2} \\
		&+\underbrace{\sum\limits_{j=1}^{2N+1}\frac{1}{\mu_j}(f-f_N^\varepsilon, \phi_j)\mathbf{d}_j}_{=:I_3}.
	\end{align*}
	
	Since $\alpha+\mu_j^2\geq 2\sqrt{\alpha}\mu_j$, we obtain that
	$$
		|I_1|^2 \leq \frac{1}{4\alpha} \sum\limits_{j=1}^{2N+1}(f^\delta-f,\phi_j)^2
		\leq \frac{1}{4\alpha} \|f^\delta-f\|^2_{L^{2}(\Gamma)}
		\leq \frac{\delta^2}{4\alpha}\|f\|_{L^{2}(\Gamma)}^2 ,
	$$
	$$
		|I_2|^2 =  \sum\limits_{j=1}^{2N+1}\left(
		\frac{\alpha}{\alpha+\mu_j^2}\right)^2\frac{(f, \phi_j)^2}{\mu_j^2}
		\leq \frac{\alpha}{4}\sum\limits_{j=1}^{2N+1}\frac{1}{\mu_j^4}(f, \phi_j)^2
		\leq \frac{\alpha}{4\mu_{\rm min}^4}\|f\|^2_{L^{2}(\Gamma)},
	$$
	and
	$$
		|I_3|^2
		=\sum\limits_{j=1}^{2N+1}\frac{1}{\mu_j^2}(f-f_N^\varepsilon,\phi_j)^2
		\leq \frac{1}{\mu_{\rm min}^2}\|f-f_N^\varepsilon\|^2_{L^{2}(\Gamma)}
		\leq \frac{(C_0 k^2\tau^{-N}+\varepsilon)^2}{\mu_{\rm min}^2}.
	$$
	Therefore,
	$$
		\left|\mathbf{c}_N^{\alpha,\delta}-\mathbf{c}_N^*\right|\leq \left(\frac{\delta}{2\sqrt{\alpha}}
		+\frac{\sqrt{\alpha}}{2\mu_{\rm min}^2}\right)\|f\|_{L^{2}(\Gamma)}
		+\frac{C_0 k^2\tau^{-N}+\varepsilon}{\mu_{\rm min}},
	$$
	which, together with \eqref{Eigenvalue}, \eqref{Thm3.2Err3}, the triangle inequality and Theorem \ref{theorem 2.1}, leads to the estimate \eqref{Thm3.2Err1}.
	
	{\bf (a)} $0< k \leq 1$. Let $\tau_0^N=|\ln\delta|^{\eta\ln \tau_0}$ , then we see
	$$
		N=N(\delta)=\eta\ln|\ln\delta|.
    $$
	Substituting $N=\eta\ln|\ln\delta|$ and the regularization parameter $\alpha=k^2\delta\tau_0^{-2N}$ into \eqref{Thm3.2Err1}, we derive the estimate \eqref{Thm3.2Err2}.
	
	{\bf (b)} $ k >1$. Let $\tau_0^N=k^{\frac{11\ln \tau_{0}}{2\ln\tau_{\rm min}}}|\ln\delta|^{\eta\ln \tau_0}$, then we see
	$$
		N=N(\delta)=\frac{11\ln k}{2\ln \tau_{\rm min}}+\eta\ln|\ln\delta|,
	$$
	which, together with $\alpha=k^{-1}\delta \tau_0^{-2N}$, leads to the estimates \eqref{Thm3.2Err21}.
\end{proof}

\begin{remark}
	To make $\tau^{-N}$ small in Theorem\ref{theorem3.2}, we choose a positive constant $\eta$ such that $N$ increases faster, since
	$\ln|\ln\delta|$ increases very slowly as $\delta$ tends to zero.
\end{remark}


\section{Numerical example}

In this section,we report an example to demonstrate the competitiveness of our algorithm. The implementation of the algorithm is based on the MATLAB software. We take $\eta=5$  and
make the assumption $\varepsilon\leq10^{-16}$ which makes $\varepsilon$ negligible compared with the discretization errors.

Since $r_{\rm in}=\min\{r_{\rm in}^{\rm max}, k^{-1}\}$ and $r_{\rm ex}>r_{\rm ex}^{\rm min}$, we have 
\begin{equation}\label{tau}
\tau_0 > \tau_{\rm min}=\frac{r_{\rm ex}^{\rm min}}{r_{\rm in}^{\rm max}}.
\end{equation}

Now, we describe our algorithm as the following:

$\bullet$ Compute $r_{\rm in}^{\rm max}$ and $r_{\rm ex}^{\rm min}$, 
and take $\tau_0$ as in \eqref{tau};

$\bullet$ Take $N$ and $\alpha$ as in Theorem \ref{theorem3.2};

$\bullet$ Solve equation $\alpha
\mathbf{c}_N^{\alpha,\delta}+A_N^{\ast} A_N
\mathbf{c}_N^{\alpha,\delta}=A_N^{\ast} b^{\delta}$.

\begin{example}\label{exm1}\rm
	To test our code, consider the case in which the exact solution to the BVP  is  $u(x)=\mathrm{e}^{\mathrm{i} k x\cdot d}$, where $d=(\frac{1}{2},
	\frac{\sqrt{3}}{2})$.  Let  $D$ be a non-convex kite-shaped domain with boundary $\Gamma$ described by the parametric representation
	\begin{equation*}
	(x_1(t),x_2(t))=(\cos t +0.65\cos 2t-0.65, 1.5\sin t),\quad 0\leq  t\leq 2\pi.
	\end{equation*}
	
	By simple calculations, it can be seen that $\max\limits_{B_t\subseteq D}t=0.923$ and $\min\limits_{D\subseteq B_t}t=1.985$. Then, we take $\tau_0=2.2$.
	
	Table \ref{tab-exm1} presents the relative errors between the numerical solution and the exact solution in domain $D$ with different noise levels. Table \ref{tab1-exm1} shows the relative errors for the approximation of $u$ and $\partial u/\partial \nu$ on boundary $\Gamma$ with different noise levels. Visually, Figure \ref{F1:exm1} shows the numerical solution for wave number $k=7$ with different noise levels. In Table \ref{tab-exm1} and \ref{tab1-exm1}, the notation $\mathrm{e}-n$ denotes the scale $\times10^{-n}$. From these tables and figures it can be seen that the numerical solution is a stable approximation to the exact one, and that the numerical solution converges to the exact solution as the level of noise decreases.

\begin{table}
	\caption{Errors in domain $D$ with different noise levels for Example \ref{exm1}.} 
		\begin{tabular}{lcccccc}
			\toprule
			Noise level & \multicolumn{3}{c}{$\frac{\left\|u_n^{\alpha,\delta}-u\right\|_{L^2(D)}}{\|u\|_{L^2(D)}}$}
			&\multicolumn{3}{c}{ $\frac{\left\|\nabla u_n^{\alpha,\delta}-\nabla u\right\|_{L^2(D)}}{\|\nabla u\|_{L^2(D)}}$} \\
			(N.L.)         &$k=0.5$                       &$k=1$                           &$k=5$                 &$k=0.5$                          &$k=1$             &$k=5$             \\
			\midrule
			$10^{-16}$ & $6.5\mathrm{e}-11$            &$3.8\mathrm{e}-10$         &$6.7\mathrm{e}-9$   & $1.6\mathrm{e}-9$            &$4.6\mathrm{e}-9$         &$2.8\mathrm{e}-8$\\
			$1\%$      & $3.4\mathrm{e}-4$             &$1.4\mathrm{e}-3$          &$9.5\mathrm{e}-3$   & $2.4\mathrm{e}-3$             &$6.3\mathrm{e}-3$        &$2.3\mathrm{e}-2$\\
			$5\%$     & $1.9\mathrm{e}-3$              &$5.3\mathrm{e}-3$          &$3.1\mathrm{e}-2$   & $9.7\mathrm{e}-3$             &$1.9\mathrm{e}-2$        &$6.8\mathrm{e}-2$\\
			\bottomrule
	\end{tabular}
	\label{tab-exm1}
\end{table}

\begin{table}
	\centering
	\caption{Errors on boundary $\Gamma$ with different noise levels for Example \ref{exm1}.} 
	\begin{tabular}{lcccccc}
		\toprule
		Noise level & \multicolumn{3}{c}{$\frac{\left\|u_n^{\alpha,\delta}-u\right\|_{L^2(\Gamma)}}{\|u\|_{L^2(\Gamma)}}$}
		&\multicolumn{3}{c}{ $\frac{\left\|\partial_{\nu} u_n^{\alpha,\delta}-\partial_\nu u\right\|_{L^2(\Gamma)}}{\|\partial_\nu u\|_{L^2(\Gamma)}}$} \\
		(N.L.)         &$k=0.5$                       &$k=1$                           &$k=5$                 &$k=0.5$                          &$k=1$             &$k=5$             \\
		\midrule
		$10^{-16}$ & $1.0\mathrm{e}-10$            &$7.8\mathrm{e}-10$         &$2.1\mathrm{e}-8$   & $3.2\mathrm{e}-9$             &$1.2\mathrm{e}-8$          &$1.1\mathrm{e}-7$\\
		$1\%$      & $1.5\mathrm{e}-3$             &$1.5\mathrm{e}-3$          &$1.7\mathrm{e}-2$   & $1.3\mathrm{e}-2$             &$8.1\mathrm{e}-2$          &$5.2\mathrm{e}-2$\\
		$5\%$     & $5.2\mathrm{e}-3$              &$7.5\mathrm{e}-3$          &$2.9\mathrm{e}-2$   & $3.0\mathrm{e}-2$             &$3.1\mathrm{e}-2$          &$7.8\mathrm{e}-2$\\
		\bottomrule
	\end{tabular}
	\label{tab1-exm1}
\end{table}

\begin{figure}
	\begin{center}
		\subfigure[]{\includegraphics[trim= 100 200 100 200, width=0.5\textwidth]{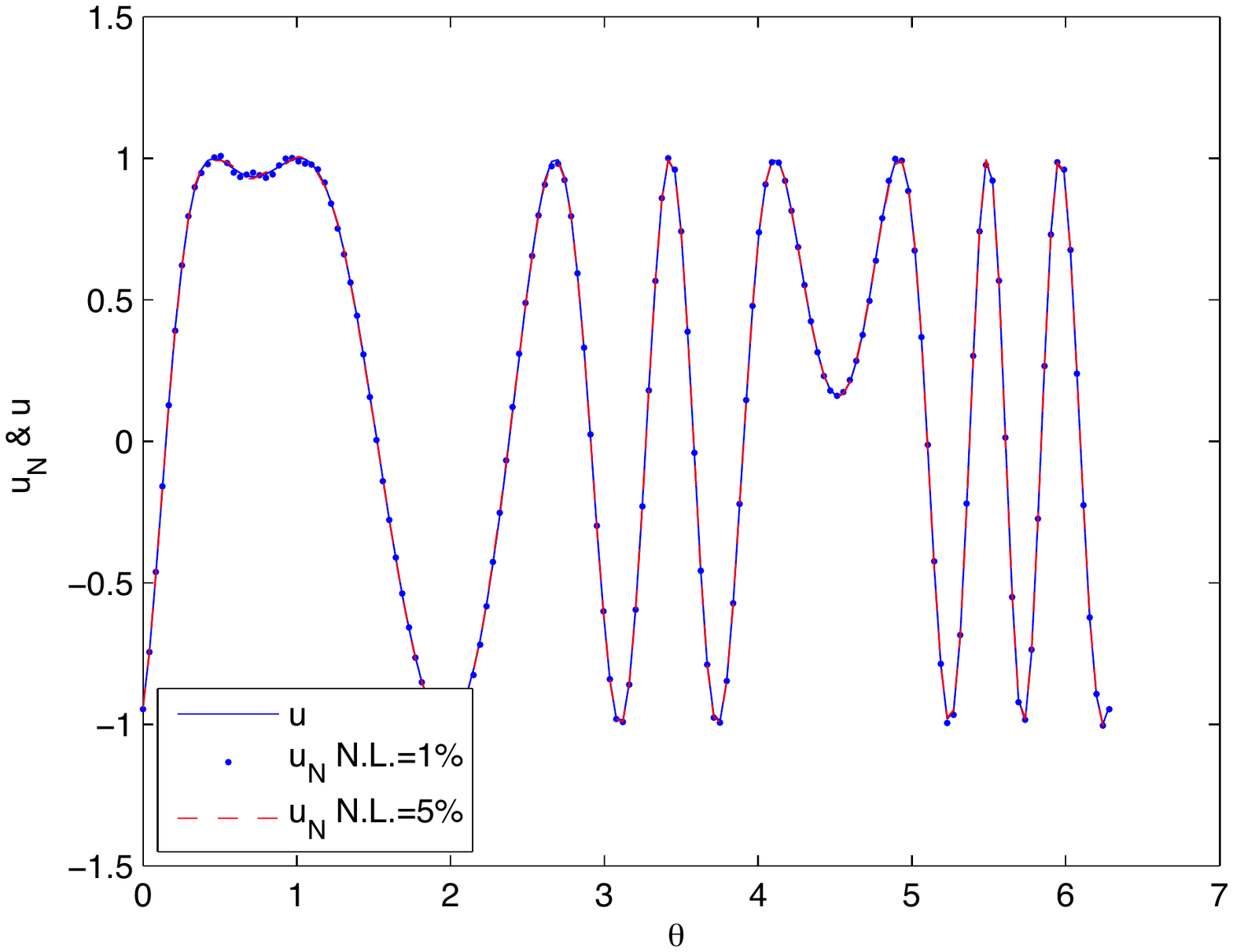}}\\
		\subfigure[]{\includegraphics[trim= 100 200 100 200, width=0.5\textwidth]{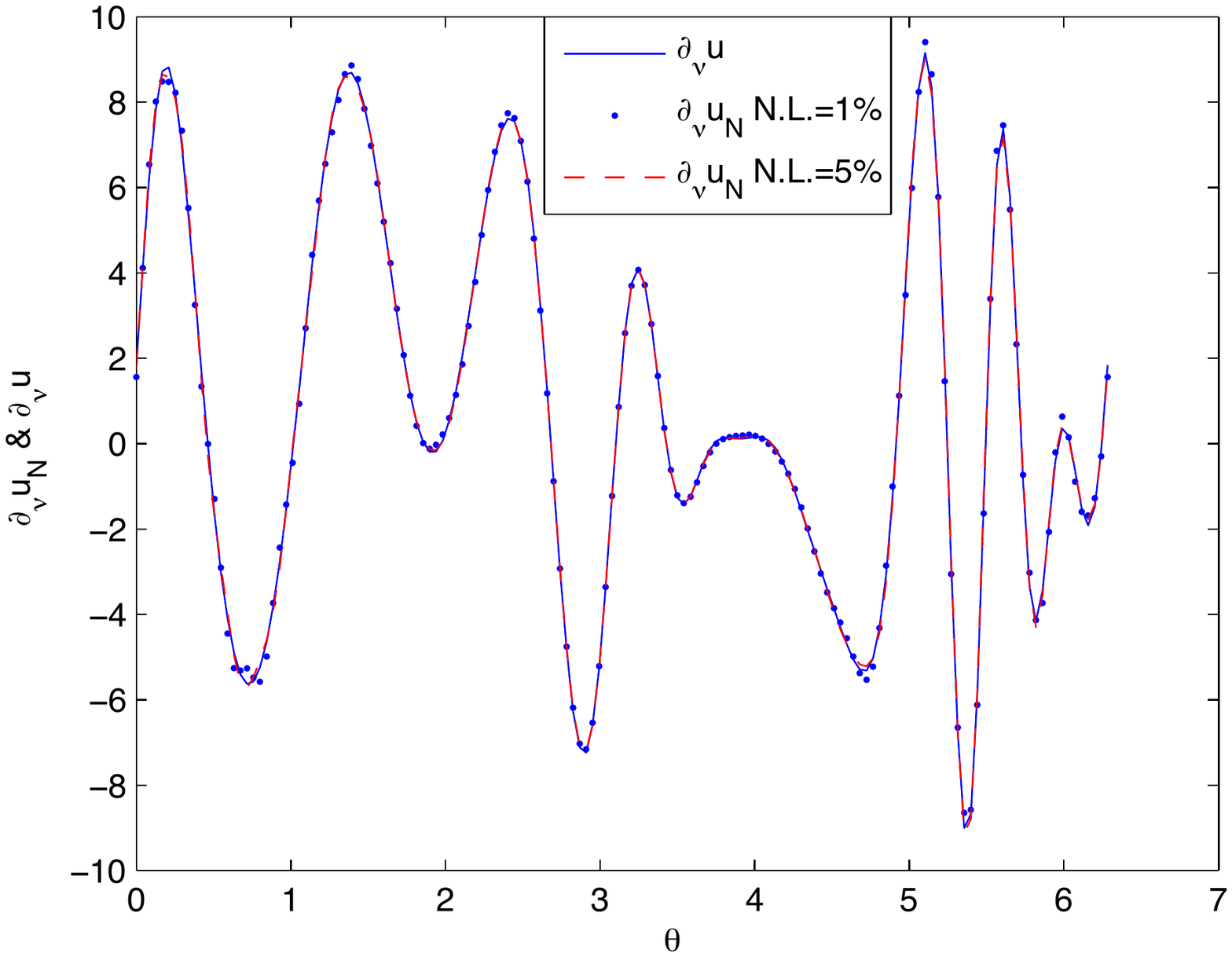}}
		\caption{The real parts of  $u$ and $u_N$ on $\Gamma$ for wave number $k=7$ with different noise levels for Example \ref{exm1}.}
		\label{F1:exm1}
	\end{center}
\end{figure}

\end{example}


\section{Conclusions}

In this paper, we study the numerical analysis for the Fourier-Bessel method to solve the BVP connected with the
Helmholtz equation.  Convergence and stability are analyzed with suitable choices of a regularization method. The method does not require
interior or surface meshing which makes it extremely attractive for solving problems under complicated boundary. We conducted some numerical experiments to show that the proposed method is stable and effective. We believe that our method should also work for high dimensional cases, and this extension is our future work.


\section*{Acknowledgements}

The research was supported by the National Natural Science Foundation of China [grant numbers 11671170, 11601107, 11671111 and 41474102].


\end{document}